\documentclass[oneside,openany,a4paper]{amsart}
\usepackage[a4paper, margin=1in]{geometry}
\usepackage[utf8]{inputenc}
\usepackage{amsfonts, amsmath, amsthm, amssymb}
\usepackage{xcolor}

\title{Non-archimedean Sendov's Conjecture}
\author{Daebeom Choi and Seewoo Lee}

\newtheorem{conjecture}{Conjecture}
\newtheorem{theorem}{Theorem}

\newtheorem{proposition}{Proposition}

\theoremstyle{remark}
\newtheorem*{remark}{Remark}
\theoremstyle{definition}
\newtheorem{example}{Example}

\DeclareMathOperator{\Qq}{\mathbb{Q}}

\DeclareMathOperator{\Dd}{\mathbb{D}}

\DeclareMathOperator{\CC}{\mathbb{C}}

\begin{document}

\maketitle

\begin{abstract}
We prove non-archimedean analogue of Sendov's conjecure. We also provide complete list of polynomials over an algebraically closed non-archimedean field $K$ that satisfy the optimal bound in the Sendov's conjecture. 
\end{abstract}

\noindent

\section{Introduction}

Sendov's conjecture \cite{sendov} can be stated as follows:

\begin{conjecture}[Sendov]
Let $f(z) = \prod_{k=1}^{n} (z-z_{k})$ be a monic polynomial over $\CC$ where all zeros are in the closed unit disk centered at zero, i.e. $|z_{k}| \leq 1$ for all $1\leq k\leq n$. Then, for each $k$, we can find a zero $w$ of $f'(z)$ such that $|w - z_{k}|\leq 1$. 
\end{conjecture}

The conjecture is known to be true for polynomials of degree $\leq 8$ \cite{degeight} and  polynomials of sufficiently large degree \cite{suflarge}, but still widely open in general.
In this paper, we study the anlogue of the Sendov's conjecture over \emph{non-archimedean} fields\footnote{We only consider characteristic 0 fields.}.
By enlarging the radius of the discs appear in the conjecture (centered at each zero $z_k$), we show that the Sendov's conjecture holds for an algebraically closed non-archimedean field $K$. 
The theorem reduces to the direct analogue of the Sendov's conjecture for polynomials of degree $n$ with norm $|n| = 1$.
At last, we give the necessary and sufficient conditions for the polynomials to meet the tight bound of the theorem.

\section{The Non-Archimedean Sendov's Conjecture.}

We first prove the following analogue of Sendov's conjecture, which has a weaker bound on the radius than the original conjecture over complex numbers. From now, let $(K, |\cdot |)$ be an algebraically closed non-archimedian field, $\mathcal{O}_K=\{x\in K\,:\, |x|\le 1 \}$ be its valuation ring, and $\mathfrak{m}_K=\{x\in K\,:\, |x|< 1 \}$ be its maximal ideal.

\begin{theorem}
Let let $r = r_n = |n|^{-1/(n-1)}$. Then the following variation of Sendov's conjecture holds for all $f(z) \in K[z]$: assume that all zeros $z_{i}$ of $f(z)$ are in the closed unit disk centered at zero. 
Then for each $z_i$, the closed disk $\overline{\Dd}(z_i, r_n):= \{w \in K\,:\, | w - z_i| \leq r_n\}$ contains at least one zero of $f'(z)$. 
\end{theorem}

\begin{proof}
Let $f(z) = \prod_{i=1}^{n} (z-z_{i}) \in K[z]$ with $|z_{i}| \leq 1$. 
Let $\{w_1, \dots, w_{n-1}\}$ be the zeros of $f'(z)$, so that $f'(z) = n\prod_{i=1}^{n-1} (z - w_{i})$.  
Assume that Sendov's conjecture is false for $f(z)$. 
Then, without loss of generality, we have $|z_{1} - w_{i}| > r_{n}$ for all $1\leq i\leq n-1$. 
This gives 
$$
f'(z) = n\prod_{i=1}^{n-1} (z-w_{i}) \Rightarrow |f'(z_{1})| = |n| \prod_{i=1}^{n-1} |z_{1} - w_{i}|  = |n|r_{n}^{n-1} >  1. 
$$
However, since $z_1$ is a zero of $f(z)$, we also have 
$$
f'(z) = \sum_{i=1}^{n} \prod_{j\neq i} (z - z_{j}) \Rightarrow |f'(z_1)| = \prod_{i=2}^{n} |z_{1} - z_{i}| \leq \prod_{i=2}^{n} \max\{|z_{1}|, |z_{i}|\} \leq 1,
$$
which gives a contradiction.
\end{proof}
\begin{remark}
Note that since $r_n\ge 1$ and $|z_i|=1$, $\overline{\Dd}(z_i, r_n)=\overline{\Dd}(0, r_n)$ for every $i$. Hence, every closed disks occur at the end of Theorem 1 are in fact all same. Therefore, the Theorem 1 is equivalent to $|w|\le r_n$ for some zero $w$ of $f'$.
\end{remark}
For example, when $K =\overline{\mathbb{Q}_{p}}$ with $p\nmid n$, we have $|n| = |n|_{p} = 1$ and so $r_{n} = 1$. 
In this case, the direct analogue of Sendov's conjecture is actually true.
However, when $n$ is a multiple of $p$, there is a counterexample for the original version.

\begin{proposition}
Let $p$ be an integer prime and $n$ be a multiple of $p$. 
Then $f_{n}(z) = z^{n} - z$ is a counterexample for the original analogue of Sendov's conjecture over $K = \overline{\Qq_p}$. 
More precisely, for all zeros $z_i$ of $f_{n}$, the closed unit disc $\mathbb{D}(z_{i}, 1) = \{w \in K\,:\, |w - z_{i} |\leq 1\}$ does not contain any zero of $f'_{n}(z)$.
\end{proposition}
\begin{proof}
Let $\zeta_{n-1}$ be $(n-1)$-th root of unity in $\overline{\Qq_{p}}$. 
Then the roots of $f_{n}(z)$ are 
$$z_j= \zeta_{n-1}^{j} \,\,(1\leq j\leq n-1), \quad z_{n} =0,$$
where the roots of $f_{n}'(z) = nz^{n-1} -1$ are 
$$w_j = n^{-1/(n-1)} \zeta_{n-1}^{j}\,\,(1\leq j\leq n-1). $$
Then we have $|z_{i} - w_{j}|_{p} > 1$ for all $1\leq i \leq n$ and $1\leq j \leq n-1$. 
In fact, we have
$$
|z_{i} - w_{j}|_p = |\zeta_{n-1}^{i} - n^{-1/(n-1)} \zeta_{n-1}^{j}|_{p} = |n|_{p}^{-1/(n-1)} \cdot |n^{1/(n-1)} \zeta_{n-1}^{i-j} - 1|_p.
$$
for $1\leq i \leq n-1$.
Since $n$ is a multiple of $p$, we have $|n|_p < 1$ and so $|n^{1/(n-1)}\zeta_{n-1}^{i-j}|_{p} = |n|_{p}^{1/(n-1)} < 1$, which gives 
$$
|z_{i} - w_{j}|_p = |n|_{p}^{-1/(n-1)} > 1. 
$$
Note that $|\alpha + \beta|_p = \max\{|\alpha|_{p}, |\beta|_p\}$ whenever $|\alpha|_p \neq |\beta|_p$. 
Similarly, for $z_n = 0$ we have
$$
|z_{n} - w_{j}|_{p} = |w_{j}|_{p} = |n|_{p}^{-1/(n-1)} > 1.
$$
\end{proof}

\section{Optimality criteria}
The counterexample given in the Proposition 1 actually satisfies the optimal bound of the Theorem 1. 
In other words, the zero $w_j$ of $f'(z)$ contained in the closed disk $\overline{\mathbb{D}}(z_i, r_n)$ actually lies in the boundary of the disk, i.e. $|z_i - w_j| = r_n$. 
In this section, we give simple criteria to determine whether a given polynomial satisfies the optimal bound of the Theorem 1. 
Let $f(z)  = z^{n} + a_{n-1}z^{n-1} + \cdots + a_{1}z + a_{0}\in K[z]$ be a polynomial where all the zeros of $f$ lies in a closed unit disk centered at origin. Note that this is equivalent to $a_i\in \mathcal{O}_K$ for every $0\le i\le n-1$ since $\mathcal{O}_K$ is integrally closed.
Define $I(f)$ as
$$
I(f) = \max_{z: f(z) = 0}\min_{w: f'(w) = 0} |z - w|. 
$$
By the Theorem 1, $I(f)\le r_n$. We will call $f(z)$ \emph{satisfies the optimal bound of the Theorem 1} if $I(f)=r_n$.
The following theorem gives necessary and sufficient conditions that a given polynomial over $K$ satisfy the optimal bound of the Theorem 1 when degree $n$ has norm strictly smaller than 1, i.e. $v(n) > 0$.

\begin{theorem}
\label{opt_poly_1}
Assume that $n=\deg f$ satisfies $|n| < 1$. Then $f(z)$ satisfies the optimal bound of the Theorem 1 if and only if
\begin{align}
\label{opt_cond}
v(a_{j}) \geq\max\left\{0, \frac{j-1}{n-1}v(n) - v(j)\right\}\,\,(2\leq j\leq n), \, v(a_1) = 0, \, v(a_{0}) \geq 0,
\end{align}
where $v = v_{K}: K \to \mathbb{R} \cup \{\infty\}$ is a valuation corresponds to $|\cdot |$.
\end{theorem}

We give two different proofs of the theorem. First proof uses the following property about Newton polygon.


\begin{proposition}
Let $K$ be a Henselian valued field with valuation $v = v_{K}$, and let $p(z) = a_{n}z^{n} + a_{n-1}z^{n-1} + \cdots + a_{1}z + a_{0} \in K[z]$.
Define $N(f)$, the Newton polygon of $f$, as a lower convex hull of the set of points $\{(i, v(a_i))\,:\, 0\leq i\leq n\}$. 
Let $\mu_1, \dots, \mu_r$ be the slopes of the line segment of $N(f)$ arranged in increasing order, and let $\lambda_1, \dots, \lambda_r$ be the corresponding lengths of the line segments projected onto the $x$-axis. 
Then for all $1\leq k\leq r$, $f(z)$ has exactly $\lambda_k$ roots of valuation $-\mu_k$.
\end{proposition}
See \cite{neukirch} for the proof of Proposition 2. 
\begin{proof}[First proof of Theorem \ref{opt_poly_1}]
First, we show that the condition is necessary. 
Let's write $f(z)$ as $f(z) = \prod_{i=1}^{n} (z-z_{i})$ and $f'(z) = n\prod_{j=1}^{n-1}(z - w_{j})$.
Without loss of generality, assume that $I(f) = |n|^{-1/(n-1)} = |z_{1} - w_{1}|$, $|z_{1} - w_{j}| \geq |n|^{-1/(n-1)}$ and $\min_{1\leq j \leq n-1} |z_{i} - w_{j}| \leq |n|^{-1/(n-1)}$ for all $i \neq 1$. 
We have
\begin{align*}
1 &\leq |n| \prod_{1 \leq j \leq n-1}|z_{1} - w_{j}| \\
&= |f'(z_{1})| \\
&= \left| \prod_{2\leq i \leq n}  (z_{1}-z_{i})\right| \\
&\leq \prod_{2\leq i \leq n} \max\{|z_{1}|, |z_{i}|\} \leq 1,
\end{align*}
and all the inequalities should be equality, so $|z_{1} - w_{j}| = |n|^{-1/(n-1)}$ for $1\leq j \leq n-1$. 
Since $|z_{1}| \leq 1$ and $|n|^{-1/(n-1)} > 1$, we get $|w_{j}| = |n|^{-1/(n-1)}$, i.e. $v(w_{j}) = -\frac{1}{n-1}v(n)$. 
From $w_{1}\cdots w_{n-1} = (-1)^{n-1}\frac{a_1}{n}$, 
$$
-v(n) = \sum_{1\leq j \leq n-1} v(w_j) = v(a_1) - v(n)
$$
and so $v(a_1) = 0$.
Now, consider the Newton polygon $N(f')$ of $f'(z) = nz^{n-1} + (n-1)a_{n-1}z^{n-2} + \cdots + 2a_{2}z + a_{1}$, which is a convex hull of $(j-1, v(j) + v(a_{j}))$ for $1\leq j \leq n$.
The slope of the segment that connects $(0, v(a_1)) = (0, 0)$ and $(n-1, v(n))$ is $\frac{v(n)}{n-1}$, which equals to $-v(w_j)$ for all $j$.
So the Newton polygon itself become the segment $\{(t, \frac{v(n)}{n-1}t)\,:\, 0\leq t \leq n-1\}$, and all other points should be located above this line, which is equivalent to  
$$
\frac{v(j) + v(a_{j})}{j-1} \geq \frac{v(n)}{n-1}\Leftrightarrow v(a_{j}) \geq \frac{j-1}{n-1}v(n) - v(j).
$$
For the Newton polygon of $f(z)$, since all zeros lie in the unit disk, $v(z_{i})\geq 0$ for all $1\leq i \leq n$. 
This means that the slopes of the line segments of the Newton polygon $N(f)$ is not positive.
Since $v(a_n) = 0$, all the points $(i, v(a_i))$ should lie above $x$-axis, i.e. $v(a_i) \geq 0$ for all $i$. Hence we get \eqref{opt_cond}.

Conversely, suppose that the coefficients of $f(z) = z^{n} + a_{n-1}z^{n-1} + \cdots + a_{1}z + a_{0} \in K[z]$ satisfy \eqref{opt_cond}.
Then the Newton polygon $N(f')$ is the line segment $\{(f, \frac{v(n)}{n-1}t\,:\, 0\leq t\leq n-1\}$ and all the zeros $w_1, \dots, w_{n-1}$ of $f'(z)$ have valuation $-\frac{v(n)}{n-1}$, i.e. $|w_j| = |n|^{-1/(n-1)}$. 
From $|z_i| \leq 1$, we have $|z_i - w_j| = |n|^{-1/(n-1)}$ for $1\leq i \leq n$, $1\leq j \leq n-1$ and so $I(f) = |n|^{-1/(n-1)}$.
\end{proof}
\begin{proof}[Second proof of Theorem \ref{opt_poly_1}]
It is enough to show that $f(z)$ satisfies the optimal bound of the Theorem 1 if and only if 
\begin{align}
\label{opt_cond_mult}
|a_j|\le \min\left\{1, \left|\frac{n^{\frac{j-1}{n-1}}}{j}  \right| \right\}\,\,(2\leq j\leq n), \, |a_1|=1, \, |a_0|\le 1.
\end{align}
As we noted above, $|z_i|\le 1$ for all $1\le i\le n$ if and only if $|a_i|\le 1$. We will prove the equivalence under this assumption.\\
For any $1\le i\le n$ and $1\le j\le n-1$, $|z_i-w_j|\le \max\{|z_i|, |w_j| \}\le \max\{1, |w_j|\}$, so if $I(f)=r_n>1$, then $r_n\le |w_j|$ for any $1\le j\le n-1$. Since 
\[ |n|^{-1}=r_n^{n-1}\le \left|\prod_{j=1}^{n-1}w_j\right|=\left|(-1)^{n-1}\frac{a_1}{n} \right|\le |n|^{-1}, \]
all the inequalities should be equality. Hence $|w_j|=r_n$ for every $1\le j\le n-1$. Also, if $|w_j|=r_n$ for all $1\le j\le n-1$, then trivially $I(f)=r_n$. Hence $f(z)$ satisfies the optimal bound of Theorem 1 if and only If $|w_j|=r_n$ for all $1\le j\le n-1$. 
Now define $h(z) \in K[z]$ as
\[ h(z)=f'(n^{-\frac{1}{n-1}}z)=z^{n-1}+\sum_{j=1}^{n-1}\frac{ja_j}{n^{\frac{j-1}{n-1}}}z^{j-1}.  \]
Then zeros of $h$ are $n^{\frac{1}{n-1}}w_j$, so $|w_j|=r_n$ for all $1\le j\le n-1$ if and only if every zero of $h$ has valuation 1. 
Since $h$ is a monic polynomial, it is easy to see that this occurs exactly when all the coefficients of $h(z)$ are in $\mathcal{O}_K$ and the constant term is in $\mathcal{O}_K^{\times}$, which is equivalent to 
\[ |a_j|\le \left|\frac{n^{\frac{j-1}{n-1}}}{j}  \right| \,\,(2\leq j\leq n), \, |a_1|=1. \]
\end{proof}

\begin{example}
We already saw that the polynomial $f_n(z) = z^n - z$ satisfy the optimal bound, and the coefficients of $f_n$ satisfies \eqref{opt_cond} (note that $v(0) = \infty$). 
\end{example}
\begin{example}
Let $p$ be an odd prime and $f(z) = z^{2p} - p^{-1/2}z^{p}$. Then $f(z)$ fails to satisfy \eqref{opt_cond} since $v(a_p) = -\frac{1}{2} < -\frac{p}{2p-1}$. Indeed, the zeros of $f(z)$ and $f'(z) = 2pz^{2p-1} - p^{1/2}z^{p-1}$ are 
\begin{align*}
    z_{1} = \cdots = z_{p} = 0, &z_{p+1} = p^{-1/2p}, \cdots, z_{2p} = p^{-1/2p} \zeta_{p}^{p-1} \\
    w_1 = \cdots = w_{p-1} = 0, &w_{p} = (4p)^{-1/2p}, \cdots, w_{2p-1} = (4p)^{-1/2p} \zeta_{p}^{p-1}
\end{align*}
and one can check that  $I(f) = p^{1/2p} < r_{2p} = p^{1/(2p-1)}$.
\end{example}

When $|n| = 1$, we express optimality condition as a non-divisibility condition of a reduced polynomial over a residue field of $K$. 
For any elements $g\in \mathcal{O}_K[z]$ (resp. $x\in \mathcal{O}_K$), let $\overline{g}$ (resp. $\overline{x}$) be the corresponding element (mod $\mathfrak{m}_K$ reduction) in $\mathcal{O}_K/\mathfrak{m}_K[z]$ (resp. $\mathcal{O}_K/\mathfrak{m}_K$).
\begin{theorem}
\label{opt_poly_2}
Assume $n=\deg f$ satisfies $|n|=1$. Then $f(z)$ satisfies the optimal bound of the Theorem 1, i.e. $I(f) = |n|^{-1/(n-1)}=1$ if and only if $\overline{f}$ does not divides $\overline{f'}^n$.
\end{theorem}
\begin{proof}
Since $f'\in \mathcal{O}_K[z]$ and the leading coefficient of $f'$ is an unit of $\mathcal{O}_K$, every zeros of $f'$ are in $\mathcal{O}_K$.\\
For any polynomial $g$, let $V(g)$ be the set of zeros of $g$. By definition, $f$ does not satisfy the optimal bound of Theorem 1 if and only if for every $z\in V(f)$ there exists $w\in Z(f')$ such that $|z-w|<1$. Since $|z-w|<1$ if and only if $\overline{z}=\overline{w}$, this is equivalent to for any $z\in V(f)$ there exist $w\in V(f')$ such that $\overline{z}=\overline{w}$. Since $\{\overline{z},\ :\, z\in V(g) \}=V(\overline{g})$ for any polynomial $g\in \mathcal{O}_K[z]$ with unit leading coefficient, this is equivalent to $V(\overline{f})\subseteq V(\overline{f'})$. By considering the linear factorization of $\overline{f}$ and $\overline{f'}$, this is equivalent to $\overline{f}\mid \overline{f'}^n$.
\end{proof}
By the long division of polynomials, the condition \emph{$\overline{f}$ does not divides $\overline{f'}^n$} can be expressed as a set of inequalities of valuations of certain polynomials whose variables are $a_1,\cdots, a_n$. Hence, in this sense, Theorem 3 has a similar spirit as Theorem 2 (although the criterion is a little more complicated).

\begin{example}
Consider $f(z) = z^{2} + z + \frac{1-p^{2}}{4} = \left(z + \frac{p+1}{2}\right)\left(z - \frac{p-1}{2}\right) \in K[z]$ with $K = \overline{\mathbb{Q}_p}$. 
Then $f'(z) = 2z + 1$ and $f'(z)^{2} = 4z^{2} + 4z + 1 \equiv 4f(z) \,(\mathrm{mod}\,\mathfrak{m}_K)$ so $f'(z)^{2}$ is divisible by $f(z)$ in $\mathcal{O}_K / \mathfrak{m}_K[z]$. 
Indeed, the distances between zeros $z = (-1\pm p)/2$ of $f(z)$ and the zero $w = -1/2$ of $f'(z)$ is $|p| = p^{-1}$, and $I(f) = p^{-1} < 1$. 

\end{example}

\end{document}